\documentclass{amsart}

\usepackage{amssymb,amsfonts,amsthm,amsmath,latexsym}
\usepackage[all]{xy}
\usepackage[T1]{fontenc}
\usepackage{mathrsfs}
\usepackage[utf8]{inputenc}
\usepackage[english]{babel}
\usepackage{amssymb}
\usepackage{amsthm}
\usepackage{graphics}
\usepackage{amsmath}
\usepackage{amstext}
\usepackage{multirow}
\usepackage{hyperref}
\usepackage{comment}
\usepackage{color}
\usepackage{tikz-cd}

\usepackage{todonotes}

\newcommand\restr[2]{{
  \left.\kern-\nulldelimiterspace 
  #1 
  \vphantom{\big|} 
  \right|_{#2} 
  }}

\newtheorem{thm}{Theorem}[section]
\newtheorem{dfn}[thm]{Definition}
\newtheorem{lemma}[thm]{Lemma}
\newtheorem{prop}[thm]{Proposition}
\newtheorem{cor}[thm]{Corollary}
\newtheorem{que}{Question}
\newtheorem{conj}[que]{Conjecture}

\newtheorem{THM}{Theorem}

\theoremstyle{remark}

\newtheorem{rem}[thm]{Remark}

\newcommand{\mb}{\mathbb}
\newcommand{\mc}{\mathcal}
\newcommand{\mf}{\mathfrak}
\newcommand{\R}{\mb R}
\newcommand{\C}{\mb C}

\newcommand{\Z}{\mb Z}
\newcommand{\Q}{\mb Q}
\newcommand{\N}{\mb N}
\newcommand{\D}{\mb D}

\newcommand{\F}{\mc F}
\newcommand{\G}{\mc G}

\DeclareMathOperator{\Aut}{Aut}

\DeclareMathOperator{\Sing}{Sing}
\DeclareMathOperator{\sing}{sing}

\DeclareMathOperator{\Ric}{Ric}

\DeclareMathOperator{\tang}{tang}
\DeclareMathOperator{\Supp}{Supp}
\DeclareMathOperator{\Red}{Red}

\numberwithin{equation}{section}

\begin{document}
\title{Numerically nonspecial varieties}

\author[Pereira]{Jorge Vitório Pereira}
\address{Jorge Vitório Pereira \\IMPA, Estrada Dona Castorina, 110, Horto\\ Rio de Janeiro,
Brasil}
\email{\href{mailto:jvp@impa.br}{jvp@impa.br}}

\author[Rousseau]{Erwan Rousseau}
\address{Erwan Rousseau \\ Institut Universitaire de France
	\& Aix Marseille Univ\\
		CNRS, Centrale Marseille, I2M\\
		Marseille\\
		France}
\email{\href{mailto:erwan.rousseau@univ-amu.fr}{erwan.rousseau@univ-amu.fr}}

\author[Touzet]{Fr\'ed\'eric Touzet}

\address{Fr\'ed\'eric Touzet \\ Univ Rennes, CNRS, IRMAR - UMR 6625, F-35000 Rennes, France.}
\email{\href{mailto:frederic.touzet@univ-rennes1.fr}{frederic.touzet@univ-rennes1.fr}}

\thanks{This work was supported by the ANR project \lq\lq FOLIAGE\rq\rq{}, ANR-16-CE40-0008  and CAPES-COFECUB Ma 932/19 project. The first author was supported by Cnpq and FAPERJ.}

\begin{abstract}
Campana introduced the class of special varieties as the varieties admitting no Bogomolov sheaves i.e.\ rank one coherent subsheaves of maximal Kodaira dimension in some exterior power of the cotangent bundle. Campana raised the question if one can replace the Kodaira dimension by the numerical dimension in this characterization. We answer partially this question showing that a projective manifold admitting a rank one coherent subsheaf of the cotangent bundle with numerical dimension 1 is not special. We also establish the analytic characterization with the non-existence of Zariski dense entire curve and the arithmetic version with non-potential density in the (split) function field setting. Finally, we conclude with a few comments for higher codimensional foliations which may provide some evidence towards a generalization of the aforementioned results.
\end{abstract}

\maketitle
\sloppy
\setcounter{tocdepth}{1}
\tableofcontents

\section{Introduction}
In \cite{Ca04}, Campana introduced the class of \emph{special} varieties as the compact K\"ahler varieties admitting no maps onto an orbifold of general type. Campana also establishes that a complex projective manifold $X$ is special if and only if it has no \emph{Bogomolov sheaf}, which are rank one coherent subsheaves of $\Omega_X^p$ having Kodaira dimension $p$  for some $p>0$. Campana conjectures that special varieties should have  properties related to Lang's conjectures on the distribution of entire curves or rational points on projective manifolds. More precisely he formulates the following conjecture.

\begin{conj}[Campana \cite{Ca04}]\label{Camconj}
 A compact K\"ahler manifold $X$ is special if and only if it contains a Zariski dense entire curve $f:\C \to X$, i.e., the image of $X$ is not contained in any proper subvariety of $X$.
Moreover, if $X$ is projective and defined over a number field $k$, then $X$ is special if and only if $X(k)$ is potentially dense, i.e., $X(k')$ is Zariski dense in $X$ for some finite extension $k' \supset k$.
\end{conj}

Recently, Wu \cite{Wu20} introduced the notion of \emph{numerical Bogomolov sheaves} as rank one coherent subsheaves $L$ of $\Omega_X^p$ having \emph{numerical} dimension $p$  for some $p>0$, and  a complex projective (or more generally compact K\" ahler) manifold $X$ is said to be \emph{numerically special} if it has no numerical Bogomolov sheaves. Campana (\cite[Remark 7.3]{Ca20}) raises the question of whether specialness is equivalent to numerical specialness. 

It is worth noticing that the existence of $L$ like above determines a distribution, namely $\mathrm{Ker}\ L$ of codimension $\geq p$, where equality holds if $p=1$, and that this distribution is actually integrable, according to a Theorem of Demailly \cite{ MR1922099}. This explain why foliations enter into the picture in the sequel (see in particular Theorem \ref*{THM:D} and Section \ref{highcodim}).

In this note, we address this problem for subsheaves of $\Omega^1_X$, proving the following result.

\begin{THM}\label{THM:A}
    Let $X$ be a compact K\"ahler  manifold admitting a rank one coherent subsheaf $L \subset \Omega^1_X$ of numerical dimension $1$. Then $X$ is not special i.e., it admits a rank one coherent subsheaf of maximal Kodaira dimension $p$ in $\Omega_X^p$ for some $p>0$.
\end{THM}

We also study the conjectural characterization of special manifolds following the above Conjecture \ref{Camconj}. Concerning the analytic characterization using entire curves, we prove the following.
\begin{THM}\label{THM:B}
Let $X$ be a compact K\"ahler  manifold admitting a rank one coherent subsheaf $L \subset \Omega^1_X$ of numerical dimension $1$. Then $X$ has no Zariski dense entire curves $f:\C \to X.$
\end{THM}

On the arithmetic side, we are not able to deal with rational points but rather we study a function field version of Campana's conjecture recently introduced in \cite{JR}. In this setting, the analogue of potential density is given by \emph{geometric specialness} as follows.

\begin{dfn} [Geometrically special varieties] A complex projective variety $X$ is \emph{geometrically-special} if, for every dense open subset $U\subset X$, there exists a smooth projective connected curve $C$, a point $c$ in $C$, a point $u$ in $U$, and a sequence of  morphisms $f_i:C\to X$   with $f_i(c) = u$ for $i=1,2,\ldots$ such that $C\times X $ is covered by the graphs $\Gamma_{f_i}\subset C\times X$ of these maps, i.e., the closure of $\cup_{i=1}^{\infty} \Gamma_{f_i}$ equals $ C\times X$.
\end{dfn}

Then the analogue of Campana's conjecture on potential density is formulated as follows.

\begin{conj}[\cite{JR}]
A complex projective variety $X$ is special if and only if it is geometrically special.
\end{conj}

In this setting, we prove the following.
\begin{THM}\label{THM:C}
    Let $X$ be a complex projective  manifold admitting a rank one coherent subsheaf $L \subset \Omega^1_X$ of numerical dimension $1$. Then $X$ is not geometrically special.
\end{THM}

One of the main ingredient in the proof of the previous results is the following statement of independent interest which adapts to compact K\"ahler manifolds previous work of the third author on projective manifolds \cite{MR3644247}.

\begin{THM}\label{THM:D}
	Let $(X,\F)$ a foliated K\"ahler manifold such that $\F$ is a
    holomorphic codimension one transversely hyperbolic foliation with quotient singularities.
    Assume  that $\F$ is not  algebraically integrable.
	Then, up to replacing $X$ by a non singular K\"ahler modification,
    there exists a morphism  $\Psi: X\to \D^N/\Gamma$ whose image has dimension $p\geq 2$ such that $\F=\Psi^*\G$ where $\G$ is one of the tautological  foliation on $ \D^N/\Gamma$. Moreover, the hyperbolic transverse structure of $\F$ agrees with the one obtained from pull-back.
\end{THM}

Towards a generalization of the preceding results to foliations with higher codimensions, we prove the following statement.

\begin{THM}\label{THM:E}
	Let $X$ be a  compact K\" ahler manifold and $\F$ be a smooth foliation of codimension $p$. If $c_1 (N_\F^*)$ is
	represented by a semi-positive $(1,1)$-form $\eta$ of constant rank $=p$, then $X$ is not special.
\end{THM}

The paper is organized as follows: in Section \ref{THF} we collect some preliminary definitions and properties of transversely hyperbolic foliations. In Section \ref{S:tautological} we state the main properties of the conormal bundle of tautological foliations on irreducible polydisc quotients. In Section \ref{numspecial}, we prove Theorem \ref{THM:D} and derive Theorem \ref{THM:A} from it. In Section \ref{ec}, we prove Theorem \ref{THM:B} on entire curves. In Section \ref{geospecialness}, we prove Theorem \ref{THM:C} on non-potential density in the (split) function field setting. Finally in Section \ref{highcodim}, we prove Theorem \ref{THM:E}.

\section{Transversely hyperbolic foliations}\label{THF}

In this section, we collect useful information about transversely hyperbolic foliations
on complex manifolds. We follow the terminology of \cite[Sections 3 and 5]{BPRT}.

\subsection{Transversely hyperbolic foliations}\label{SS:THF}
Let $\F$ be a codimension one foliation on a complex manifold $X$. The foliation $\F$ is transversely hyperbolic
if the sheaf of holomorphic first integrals $\mathcal O_{X/\F}$ admits a locally constant subsheaf of sets $\mathcal I$ (called the sheaf of distinguished first
integrals) such that
\begin{enumerate}
	\item every $f \in \mathcal I$ is non-constant and has image contained in the unit disk $\D$;
    \item for every non-empty, connected, and  simply-connected open subset $U$, $\mathcal I(U)$ is non-empty and equal to $\Aut(\D) \cdot f$ for any $f \in \mathcal I(U)$;
    \item if $f \in \mathcal I(U)$, $g \in \mathcal I(V)$, and $U\cap V$ is a connected open set then
    there exists $\varphi \in \Aut(\D)$ such that $\varphi \circ f = g$.
\end{enumerate}

The pull-back of the Poincar\'e metric on the unit disk by any local distinguished first integral $f \in \mathcal I$
is a closed semipositive smooth $(1,1)$-form
\[
    \eta=f^*\left( \frac{i}{\pi} \frac{du\wedge d \overline{u}}{(1-|u|^2)^2}\right)= -\frac{i}{\pi} \partial\bar{\partial} \left( \log (1-{ |f|}^2)\right)
\]
that does not depend on the choice of $f$.
If $\omega$ is a local generator of $N^*_{\F}$ then the $(1,1)$ form
\[
    \eta = \frac{i}{\pi} \exp(  2 \psi)  \omega \wedge \overline \omega
\]
defines by duality a (singular) metric on the conormal bundle $N_\F^*$
with a plurisubharmonic continuous local weight $\psi=-\log (1-{ |f|}^2) +\log( |g|)$ where $g$ is a holomorphic function such that $df=g\omega$.
Therefore, the curvature form of the induced metric on $N^*_{\F}$ is  (in the sense of currents)
\begin{equation}\label{E:curvature equation}
    T= \frac{i}{\pi} \partial \overline {\partial} \psi =   \sum_{D\in \mathcal P} m_D [D] +\eta,
\end{equation}
the (locally finite) sum being taken over the set $\mathcal P$ of prime divisors. Here, $m_D\ge 0$ denotes the ramification order of the local distinguished first integrals along $D$. In other words, a distinguished first integral at a neighborhood of a general point of $D$ is of the
form $z^{1+m_D}$ where $\{z=0\}$ is a suitable local defining equation of $D$. In particular, $T=\eta$ if, and only if, the \emph{developing map} $\tilde{\varphi}: \tilde X\to \D$ is a submersion in codimension one.
The divisor $\sum_{D\in \mathcal P} m_D D$ is the \emph{ramification divisor} of the transversely hyperbolic foliation $\F$. Note in particular that $D$ is $\F$-invariant whenever $m_D >0$. Note also that $T$ is a closed positive $(1,1)$-current, whence the

\begin{prop}
    If $X$ is a complex compact manifold and $\F$ is a transversely hyperbolic foliation on $X$ then
    $c_1(N^*_{\F})$ is pseudo-effective.
\end{prop}

\subsection{Pull-back of transversely  hyperbolic structure}
Let $f:X\to Y$ a holomorphic map between complex manifold and let $\F$  be a transversely hyperbolic foliation on $Y$. One can define the pull-back foliation $f^*\F$ provided that the image of the differential $df$ is not tangent to $\F$.  In this case, if $\F$ carries a transverse hyperbolic structure with sheaf of distinguished first integrals $\mathcal I$ then $f^*\F$ carries a transverse hyperbolic structure defined by the sheaf of distinguished first integrals  $f^*\mathcal I$.

\subsection{Transversely hyperbolic foliations with quotient singularities}\label{SS:quotientsing}
We will say that a codimension one foliation $\F$ is a transversely hyperbolic foliation with poles
if there exists a hypersurface $H$ such that $\restr{\F}{X-H}$ is a transversely hyperbolic
foliation as defined in Subsection \ref{SS:THF}.

According to \cite[Corollary 5.3]{BPRT}, any transversely hyperbolic foliation defined on $X-H$ where $H$ is an hypersurface, extends through $H$ as a foliation. Moreover, \cite[Theorem 5.2]{BPRT} describes the degeneracies of the transverse hyperbolic structure along $H$.

In this work, we are interested in the following subclass of the class of transversely hyperbolic foliations with poles.

\begin{dfn}
    A transversely hyperbolic foliation with quotient singularities on a complex compact manifold $X$ consists
    of a reduced  divisor $H = \sum { H_i}$ (the divisor of poles of the transverse structure) 
    and a transversely hyperbolic foliation $\F$
    on $X-H$ such that for any $x \in H$, there exists a neighborhood $U_x\subset X$ of $x$, such that the
    local monodromy of the transverse hyperbolic structure $\pi_1(U_x-H) \to \Aut(\D)$ is non-trivial and has
    finite image.
\end{dfn}

Let $U_x$ be as above. Every finite subgroup of $\Aut(\D)$ is cyclic and generated by an elliptic transformation. Therefore, there is no loss of generality in assuming that the image of the local monodromy representation takes values in $S^1$. Let $f$ be the corresponding multivalued distinguished first integral. Then $f^n: U_x-H\to \D$ is well defined and extends through $H$ by boundedness as a holomorphic function $g:U_x\to \D$. This implies that, on $U_x$, $f$ takes the form
\begin{equation}\label{E: firstintegral}
    f=f_x f_1^{ \nu_1} \cdots  f_r^{\nu_r}
\end{equation}
where the $\nu_1, \ldots, \nu_r$ are positive  non integral rational numbers,  $f_1 \cdots f_r=0$ is a local defining equation of $H$ and $f_x$ is a holomorphic function. One can moreover assume, up to adding a non negative integer exponents to the $\nu_i$'s that $f_x$ is not identically zero on each branch $H_i\cap U_x$. Remark also that $H$ is necessarily $\F$-invariant.	

The local expression for $f$ makes clear that the pull-back of the Poincaré metric on $\D$ by any multivalued distinguished first
integral for $\restr{\F}{U_x-H}$ extends through $H$  as the closed positive current
\[
    \eta_x= -\frac{i}{\pi} \partial\bar{\partial} \left( \log (1-{ |f|}^2)\right)
\]
If one consider the well (locally) defined logarithmic form

$\xi=\dfrac{df}{f}= \sum_i\nu_i \dfrac{df_i}{f_i} + \dfrac{df_x}{f_x}$, one can rewrite
\[
\eta_x= \frac{i}{\pi} \left(\dfrac{{ |f|}^2} {  { (1-{ |f|}^2)}^2}\right)\xi\wedge \bar{\xi}
\]

Set $g=f_x\prod_i f_i$. The holomorphic form $\omega= g\xi$  has no zeroes in codimension one (see proof  of Proposition \ref*{P:divhyper} ), hence is is a local generator of $N^*_{\F}$ in a neighborhood of $x$. One can then readily check that

 \[\eta_x = \frac{i}{\pi} \exp(  2 \psi)  \omega \wedge \overline \omega
\]
with
$$\psi= -\log (1-{ |f|}^2) +\sum_{i=1}^r ( \nu_i-1) \log( |f_i|) +\sum_{j=1}^s(m_j-1)\log( |f_ { x,j})$$ and where $f_x=\prod_{j=1}^{s} f_ { x,j}^{m_j}$ is the writing of $f_x$ as a product of irreducible factors.
By construction these local $(1,1)$-forms glue together with that defined on  $X-H$ by (\ref{E:curvature equation}) and then give rise by duality to a global singular metric  $N^*_{\F}$ with local weight $\psi$. Observe that  $\eta$, considered as a closed positive current, has  as a $(1,1)$
continuous plurisubharmonic potential of the form $\varphi=  -\log (1-{ |f|}^2)$ so that its Lelong numbers  $\nu(\eta,x)=0$ at any $x\in X$. When $X$ is compact, Demailly's approximation Theorem \cite[Theorem 4.1]{MR1158622} implies that $\eta$ represents a $\emph{nef}$ class in $H_{\partial\bar{\partial}}^{1,1}(X,\R)$.
Observe also that $\eta$ is nothing but the unique  positive current giving no mass to $H$ and extending the semipositive form $\eta_{|X-H}$. In particular, $\eta$ coincides with its absolute continuous part with respect to the Lebesgue measure.

Let us explicitize the curvature current $T$ of the singular metric defined by $\eta$ on $N_\F^*$ in restriction to $U_x$. A straightforward calculation yields:
\[
    T_{ |U_x}=\frac{i}{\pi} \partial \overline \partial \psi = \eta + \sum_{i=1}^r (\nu_{i} - 1 ) [H_i]_{ | U_x}+ [D_x]\, ,
\]
where $D_x$ is an integral effective divisor whose support lies in the polar locus of the logarithmic derivative $\frac{df_x}{f_x}$. More precisely, if $D=D_1+\ldots+D_s$ is the divisor of poles of $\frac{df_x}{f_x}$ with corresponding residues  $m_1,\ldots,m_s$ then $D_x=\sum_{i=1}^s (m_i-1) D_i$.

The discussion above is summarized in the following result.

\begin{prop}\label{P:conormal}\label{P:c1dec}
  Let $X$ be a complex  manifold and let $\F$ be a  transversely hyperbolic foliation with quotient singularities on $X$.
  Let $H$ be the divisor of poles of the transverse structure.
  Consider the (singular) metric on $N_\F^*$ defined by $\eta$, the trivial extension through $H$ of the pull-back of the Poincar\'e metric by local distinguished first integrals. Let $T$ its curvature current (in particular $T$ represents  $c_1(N^*_{\F})\in H_{\partial\bar{\partial}}^{1,1} (X,\R) )$, then
  \begin{equation}\label{E:summation}
        T =\sum_{ D\in\mathcal P} r_D [D] +\eta
  \end{equation}
   where $D$ ranges over the set $\mathcal P$ of prime divisors,   $r_D \in \mathbb \Q_{>-1}$ and the sum is locally finite.    Moreover,
  \begin{enumerate}
    \item the current $\eta$   is a smooth semi-positive $(1,1)$-form in restriction to $X-H$ and when $X$ is compact, represents a nef class in $H_{\partial\bar{\partial}}^{1,1}(X,\R)$; and
    \item if $r_D\not=0$, $\F$  admits, at a general point of $D$, a distinguished (maybe multivalued) first integral of the form $z^{r_D +1}$ where $z=0$ is a local defining equation of $D$; and
    \item the set $\{D\in\mathcal P \arrowvert r_D\notin\Z\}$ coincides with the set $\{H_i, i\in I\}$.
  \end{enumerate}
\end{prop}

\begin{dfn}\label{D:divisorial part}
	The divisor  $\sum_{ D\in\mathcal P} r_D D$ will be called the \textbf{divisorial part} of $\F$
    (with respect to the given transversely hyperbolic structure).
\end{dfn}

\begin{rem}
	The decomposition presented in Equation (\ref{E:summation}) is compatible with restriction to open subsets $U$
    (where the tranverse hyperbolic structure of $\restr{\F}{U}$ is given by restriction of the sheaf $\mathcal I$). In particular,
    the divisorial part of the restricted transverse structure is just $\sum_{ D\in\mathcal P} r_D  \restr{D}{U}$.	
\end{rem}

\subsection{Divisorial part along invariant hypersurfaces}
Let $\F$ be a transversely hyperbolic foliation with quotient singularities on a complex manifold $X$.
Let $H=\sum H_i$ be its divisor of poles. Let us denote by $\mathcal I_{d \log}$ the sheaf defined on $X-H$ by the collections of logarithmic differential $df/f$ where $f\in \mathcal I$, see also \cite[D\'efinition 5.3]{MR3124741}.

\begin{prop}\label{P:divhyper}
	Let $K$ a hypersurface of $X$. Assume  there exist a neighborhood $U$ of $K$
    and a section of $\mathcal I_{d \log}$ on $U-K$ which extends through $K$ as a logarithmic one form $\omega$ such that $K\subset { (\omega)}_\infty$.
	The following assertions hold true.
	\begin{enumerate}
		\item\label{I:divhyper1} The irreducible components of the hypersurface $K$ are $\F$-invariant.
		\item\label{I:divhyper2} The germ of $\omega$ along $K$ is unique.
		\item\label{I:divhyper3} If $D$ is a prime divisor of $U$, then the residue  $\lambda_D$ of $\omega$ along $D$ belongs to $\Q_{\geq 0}$.
		\item\label{I:divhyper4} Up to shrinking $U$, $\omega$ has no zeroes in codimension one and
        the divisorial part of $\F_{|U}$ is $\sum_{ D\in\mathcal P} r_D D$
        where $r_D= 0$ if $\lambda_D=0$, and $r_D= \lambda_D- 1$ otherwise.
\end{enumerate}
\end{prop}
\begin{proof}
	Item (\ref{I:divhyper1}) is obvious. Indeed, $K$ is a component of the polar locus of a closed meromorphic form defining the foliation on $U$.

    Let $x\in K$. If $x\notin H$, then there exists in the neighborhood $U_x$ of $x$ and a  section $f$ of  $I$ over $U_x-K$
    such that $\omega=\frac{df}{f}$. As noticed before, $f$ extends through $K$ as a section of $\mathcal I$ over $U$. As $K\subset { (\omega)}_\infty$, one necessarily has $f(x)=0$. Hence, this section is unique modulo multiplication by a complex number of modulus one. Consequently, $\omega$ is unique in restriction to $U-H$.
	
	If $x\in H\cap K$, there exists a neighborhood $U_x$ of $x$ and a multivalued distinguished first integral $f$ on $U_x-K$
    with finite and non trivial multiplicative monodromy taking values in $S^1$. The uniqueness of $\omega=\frac{df}{f}$ follows from the observations already made in Subsection \ref{SS:quotientsing}. 	This establishes the uniqueness stated in Item (\ref{I:divhyper2}).

    Let $F=f_1\cdots f_r \cdot f_{r+1}\cdots f_p=0$ be a local reduced equation for the polar locus of $\omega$
    in a small neighborhood $U_x$ of $x\in K$, where $f_1\cdots f_r=0$ is a local equation for $H$.
    By construction, there exists $\nu_1,\ldots,\nu_r\in{\mathbb Q}_{>0}$, $m_{r+1},\ldots,m_{p}\in \N_{>0}$, such that
	$$
        \omega_{|U_x}= \sum_{i=1}^r \nu_i\frac{df_i}{f_i} +\sum_{i=r+1}^p {m_i}\frac{df_i}{f_i} +\omega_0
    $$
	where $\omega_0$ is some holomorphic one form.  In particular, the property mentioned in Item \ref{I:divhyper3} is satisfied.
	
    Equivalently, $\F$ admits on $U_x$ a multivalued distinguished first integral of the form $e^{\int \omega}=f=uf_1^{\nu_1}\cdots f_r^{\nu_r} f_{r+1}^{m_{r+1}}\cdots f_p^{m_p}$ where $u$ is a unit. As before, this enables to compute the divisorial part of
	$\restr{\F}{U_x}$, namely $\sum_{i=1}^r( \nu_i -1) D_i+\sum_{i=r+1}^p (m_i-1)D_i$ where $D_i=\{f_i=0\}$.
	This proves  the second assertion of Item (\ref{I:divhyper4}).
	
    By considering the well defined real first integral $g=|f|$,
    one remark that there is no invariant hypersurface passing through $x$ except the poles $D_i$. In particular the germ of $\omega$ along $K$ has no zeroes in codimension one. This establishes the first point of Item (\ref{I:divhyper4}).	
\end{proof}

\subsection{Behavior under pull-back by a surjective morphism}\label{SS:behavior under}

If $\varphi:X\to Y$ is a surjective morphism between complex compact manifolds and $\F$ is a tranversely hyperbolic
foliation with quotient singularities on $Y$, the pull-back foliation $ \varphi^*\F$  carries also a  transversely hyperbolic structure with quotient singularities directly inherited from that of $\F$, that is induced on $X-\varphi^{-1} (H)$ by the sheaf $\varphi^*{\mathcal I}$.
From Proposition \ref*{P:c1dec}, one obtains a decomposition of
$N_\F^*$ which reads (in $\mathrm{Pic} (Y)\otimes\Q$) as
\[
  N_\F^*= L+D
\]
where $D$ is the divisorial part of $\F$ (see Definition \ref{D:divisorial part}), $L$ is a nef $\Q$ line bundle whose Chern class is represented by $\eta$. A similar decomposition holds for the conormal sheaf of $\varphi^* \F$. Both decompositions are indeed naturally related as shown by the next result.

\begin{prop}\label{P:behaviorsurj}
	With assumptions and notations as above,
	\[
        N_{\varphi^* \F}^*= \varphi^* (L) + D'
    \]	
	where $D'$ is the divisorial part of $\varphi^* \F$.
\end{prop}
\begin{proof}
    Let $\G=\varphi^*\F$. On $X$, we have a $\G$ invariant divisor $I$ which, roughly speaking, is the locus where $\varphi$ ramifies over the direction tranverse to $\F$. More precisely, if $\omega$ is a generator of $N^*_{\F}$ on an open subset $U$ , the restriction of $I$ to $\varphi^{-1}(U)$ is the zeroes divisor of $\varphi^*\omega$.
    The line bundles $\varphi^* N_\F^*$ and $ N_\G^*$ are related by the equality
    $N_\G^*=\varphi^* N_\F^* +I.$

    Then, we have just to verify that $D'=\varphi^*D +I$. It suffices to show that equality
    \begin{equation}\label{E:localeq}
        D'_{|\varphi^{-1}(U)} =\varphi^*D_{|U} +I_{|\varphi^{-1}(U)}
    \end{equation}
    holds for every member $U$ of an open cover ${(U)}_{U\in \mathcal U}$  of $Y$.
    First, let $x\in Y-\mathrm{Supp}(D)$. In some neighborhood $U$ of $x$, $N_\F^*$ is generated by $df$ where $f\in \mathcal I (U)$ and Equation (\ref{E:localeq}) is obviously true.

    If $x\in \Supp(D)$, let $D_1,\ldots,D_q$ be the components of $\Supp(D)$ such that $x\in D_i, i=1,\ldots,q$ are ordered such that $r_{D_i}\notin \N$ for $i=1,\ldots,p$, $r_{D_i}\in\N_{>0}$ for $i=p+1,\ldots,q$. According to Subsection \ref{SS:quotientsing}, $\F$ is defined in some  small neighborhood $U$ of $x$ by a closed logarithmic form $\xi= \sum_{i=1}^q (r_{D_i}+1) \frac{df_i}{f_i} + \sum_{i=q+1}^s \frac{df_i}{f_i}$ where $f_i=0$ is a local reduced equation of $D_i$ and $h_i=0$ are additional poles with residues equal to one. Moreover $\restr{\xi}{U_x-\bigcup_{i=1}^p D_i}$ is a section of $\mathcal I_{d\log}$. Set $D_i=\{f_i=0\}$ for $i>q$. Note that $\omega= \varphi^*\xi$ is a closed logarithmic form on $\varphi^{-1} (U_x)$ fulfilling the hypothesis of Proposition \ref{P:divhyper}, with $K=\varphi^{-1}( \bigcup D_i)$ (in restriction to $\varphi^{-1} (U))$. Item (\ref{I:divhyper4}) of Proposition \ref{P:divhyper} determines the divisorial part of $\restr{\G}{\varphi^{-1}(U)}$. An elementary calculation yields
    \[
        D'_{\arrowvert\varphi^{-1} (U)}=\restr{\left(\sum_{i=1}^q r_{D_i} \varphi^*(D_i) +\sum_{i=1}^s (\varphi^* (D_i)-\varphi^* (D_i)_{red}) \right)}{\varphi^{-1}(U)}.
    \]

    On the other hand, according to the first part of Item (\ref{I:divhyper4}) of Proposition \ref{P:divhyper},   $f\omega$ is a local generator of $\F$ on $U$, where $f=\prod_i f_i$. This implies that
    \[
        \restr{I}{\varphi^{-1}(U)}=     \sum_{i=1}^s \restr{( (\varphi^* (D_i)-\varphi^* (D_i)_{red}))}{\varphi^{-1}(U)} \, ,
    \]
    thus proving the Equality (\ref{E:localeq}).	
\end{proof}

We also state the following two lemmas for further use.

\begin{lemma}\label{L: Pullbackbehavior}
	Let $\varphi:X\to Y$ be a surjective morphism  with connected fibers between compact complex manifolds.
    Let $\G$ a transversely hyperbolic foliation on $X$ with quotient singularities. Assume that there exists on $Y$ a codimension one holomorphic foliation $\F$ such that $\G=\varphi^*\F$. Then $\F$ carries a transversely hyperbolic with quotient singularities structure. Moreover the pull-back of this structure via $\varphi$ coincides with that of $\G$ wherever defined.
\end{lemma}	
\begin{proof}
    Let $H$ be the divisor of poles of $\G$. As $H$ is necessarily $\G$ invariant, the restriction of $\varphi$ to $H$ is not surjective. Therefore, there exists a non-empty open Zariski subset $U$ of $Y$ such that $\G$ is transversely hyperbolic (without poles) in restriction to $V:=\varphi^{-1} (U)$. In addition, one can suppose that $\restr{f}{V}$ is a smooth morphism onto $U$. Let $(W)$ be a covering of $U$ by open subsets (in the euclidean topology) such that the sheaf $\mathcal I$ of distinguished first integrals of $\G$ is constant on $\varphi^{-1}(W)$. The fibers being compact submanifolds, every global section of ${ \mathcal I}_{\varphi^{-1}(W)}$ descends to $W$. Consequently, $\F$ admits on $U$ a transversely hyperbolic structure defined by the locally constant sheaf $\mathcal J$ such that $\mathcal I=\varphi^*\mathcal J$. Consider the analytic subset  of $Y$ defined by $Z= Y-U$. The transverse hyperbolic structure defined by $\mathcal J$ extends through $Z-K$ where $K$ is the union of codimension one components of $Z$ around which the local monodromy is non trivial. Let $Z_0$ be a component of $Z$. Pick a general point $p$ of $Z_0$ and let $\gamma$ be a loop around $Z_0$ in some neighborhood $V_p$ of $p$.  Let $q\in \varphi^{-1} (p)$. Let $\D_q$ be a small disk centered at $q$ such that $\D_q-\{q\} $ is transverse to $\G$ and $\varphi ( \D_q )-\{p\} $ is transverse to $\F$. Let $\varepsilon :[0,1]\to  \D_q -\{q\}$ a small loop of index one around $q$. Obviously, $\varphi (\varepsilon)$ is a loop freely homotopic to a non zero multiple of $\gamma$ in $V_p-Z_0$. Because $\G$ has quotients singularities,  this implies that the local  monodromy representation along $\gamma$ has finite image, whence the result.
\end{proof}

\begin{lemma}\label{L:foldescent}
	Let $\varphi:X\to Y$ be a surjective morphism  with connected fibers between compact complex manifolds.
    Let $\G$ a transversely hyperbolic foliation on $X$. Denote by $\rho$ its monodromy representation. Assume that there exists a representation $\rho': \pi_1(Y)\to \Aut (\D)$ such that $\rho=\varphi^* \rho'$. Then there exists on $Y$ a transversely hyperbolic foliation $\F$ such that $\G=\varphi^*\F$ and whose monodromy representation is $\rho'$.
\end{lemma}
\begin{proof}
    Let $U\subset Y$ a non-empty open Zariski subset such that $\varphi$ restricts to a smooth morphism on $\varphi^{-1} (U)$. By assumption on the representation $\rho$, the sheaf $\mathcal I$ of distinguished first integrals of $\G$ is globally constant  over $W$, where $W$ is any simply connected open subset of $U$. Like before, this implies that any section $s\in \mathcal I (\varphi^{-1}(U))$ is constant on the fibers of $\varphi$. Consequently,  there exists  on $U$ transversely hyperbolic foliation $\restr{\F}{U}$, which then extends as a foliation $\F$ on the whole $Y$ (as recalled in Subsection \ref{SS:quotientsing}) and whose monodromy representation is given by the composition morphism:
    \[
        \pi_1 (U)\to \pi_1(Y)\xrightarrow{\rho'} \Aut(\D) \, .
    \]
    Since the first arrow is surjective,  the result follows.
\end{proof}

\begin{rem}\label{R:KLT}
    Let $\F$ be a tranversely hyperbolic foliation with quotient singularities on a compact complex manifold $X$.  Then there exists a smooth modification $\pi:\hat X \to X$ obtained by a sequence of successive blows-up with smooth center such that the divisorial part of $\pi^*\F$ is supported on an invariant  normal crossing divisor $D=D_1+\ldots+ D_r$ . In particular there exists a $r$ uple of rational number $(\lambda_1,\ldots,\lambda_r)\in\Q_{<1}^r$ such that $E:=N_ { \pi^*\F}+\sum_i \lambda_i D_i$ is a  pseudoeffective $\Q$ line bundle whose Chern class is represented by a non trivial positive $(1,1)$-form $\eta$ with $L_{loc}^1$ coefficient (and actually smooth on a Zariski dense open subset).  When $X$ is compact  K\"ahler,  $\pi^*\F$ is a particular case of a KLT foliation in the terminology of \cite[Section 8.1]{MR3644247}. Note also that the existence of $\eta$ guarantees that the positive part of $L$ in its Zariski decomposition is non trivial.  	
\end{rem}

\subsection{Uniqueness of the transverse structure}

The following result is established in \cite[Corollary 5.6]{BPRT}.

\begin{prop}\label{P:uniqueness}
	Let $\F$ a transversely hyperbolic foliation (with quotient singularities) on a projective manifold.
    Assume that $\F$ is not algebraically integrable. Then, the hyperbolic transverse structure is unique, i.e., any transverse hyperbolic structure for $\F$  on a dense Zariski subset is defined by the same sheaf of distinguished first integrals.
\end{prop}

\subsection{Relationship with numerical properties of the conormal bundle}
The Theorem below is essentially proved in \cite{MR3124741} and describes the interplay between the existence of a transverse hyperbolic structure and  positivity properties of the conormal bundle of a foliation. The following is essentially a reformulation of  \cite[Th\'eor\`eme 1 and Proposition 5.1]{MR3124741}, see also   \cite[Section 3.2]{MR3644247},  where it is recalled that the coefficients $r_D$
appearing in Equation (\ref{E:summation}) coincide with the coefficients of the divisorial Zariski decomposition of
$c_1(N^*_{\F})$ and must be, therefore, non-negative.

\begin{thm}\label{TH:conormalvshyperbolic}
    Let $\F$ a codimension one foliation on a compact Kähler manifold $X$ equipped with a K\"ahler form $\Theta$. Assume that $N_\F^*$ is pseudo-effective with numerical dimension one. Let $N=\sum_{i=1}^r \lambda_i N_i$ be the negative part in the Zariski decomposition of $c_1(N_\F ^*)$.	
    Then
    \begin{enumerate}
	   \item The coefficients $\lambda_i$ are positive rational numbers.
       \item\label{I:conormal2} The intersection matrix $m_{ij}=N_i\cdot N_j \cdot \Theta^ {n-2}$ is negative definite.
	   \item $\F$ admits a transverse hyperbolic structure with quotient singularities on $X$ such that
	   \begin{enumerate}
            \item the divisor of poles is $H= \sum\beta_i N_i$, where $\beta_i=0$ if $\lambda_i\in\N$, $\beta_i=1$ otherwise; and
	        \item the divisorial part of $\F$ with respect to the given transversely hyperbolic structure  is $N$.
       \end{enumerate}
    \end{enumerate}	
\end{thm}

\begin{proof}
    Everything has already been established in 	\cite{MR3124741, MR3644247} except the rationality of the coefficients $\lambda_i$'s when the ambient manifold is K\"ahler but non projective. Let us justify this property. According to \cite[Proposition 5.1]{MR3124741}, there exists in a small neighborhood $U$ of $\mathrm{Supp}(N)$ a closed logarithmic form  $\omega$ defining the foliation on $U$  (which restricts to a section of ${ \mathcal I}_{d\log)}$ on $U-\mathrm{Supp}(N)$) with the following additional properties:
    \begin{enumerate}
        \item The divisor of poles of $\omega$ has the following form:
        \[
            {(\omega)}_\infty=\sum N_i +A
        \]
        where $A$ is a hypersurface of $U$ intersecting $\mathrm{Supp}(N)$ along a codimension two subset.
        \item $\mathrm{Res}_{N_i} \omega=\lambda_i +1$, $\mathrm{Res}_{A} \omega=1$
        \item $\omega$ has no zeroes in codimension one.
    \end{enumerate}	
    As an immediate consequence,  the real Chern classes class of $N^*_{\F}$ and of $N$ coincide in $H^2(U,\mathbb R)$. On the other hand, the class of $N^*_{\F}$ lies in $H^2(U,\mathbb Q)\subset H^2(U,\mathbb R)$  as the class of any line bundle.

    Suppose by contradiction that at least one of the coefficients $\lambda_i$ lies in $\R-\Q$. By rationality of $c_1(N_i)$, one promptly deduces that there exists $(\nu_1,\ldots,\nu_r)\in\R^r-\{0\}$ such that $\sum_i \nu_i c_1 (N_i)=0$ in $H^ 2(U,\R)$. Now, by de Rham's isomorphism, $c_1(N_i)$ can be represented on $U$  by a real closed two form $\theta_i$ and the linear dependance relation above is equivalent to the fact that $\theta:=\sum \nu_i \theta_i$ is exact on $U$. Let us evaluate the intersection product $I={ \left( \sum \nu_i c_1\left( N_i\right)\right)}^2 \Theta^{n-2}$. By Item (\ref{I:conormal2}) of the Theorem, it is a negative real number but one can alternatively compute this intersection as $I=\sum \nu_i\int_{N_i}\theta\wedge\Theta^{n-2} =0$ by exactness of $\theta$, whence the contradiction.
\end{proof}

\section{Pull-backs of tautological foliations on irreducible polydisc quotients}\label{S:tautological}

\subsection{Irreducible polydisc quotients}
Let $N\ge 2$ be an integer. A discrete subgroup $\Gamma \subset \Aut(\mathbb D)^N$ is a lattice if
the quotient $\mathbb D^N/\Gamma$ has finite volume. A lattice $\Gamma \subset \Aut(\mathbb D)^N$ is irreducible
if it is not commensurable to a product of $\Gamma_1 \times \Gamma_2 \subset \Aut(\mathbb D)^{N_1} \times \Aut(\mathbb D)^{N_2}$
with $N_1, N_2 \ge 1$, $N_1+N_2=N$.

If $\Gamma \subset \Aut(\mathbb D)^N$ is an irreducible lattice then the quotient $\mathbb D^{N}/\Gamma$ is a singular
variety with finitely many cyclic quotient singularities according to  \cite[Theorem 2]{Shimizu}.

The quotient $\mathbb D^{N}/\Gamma$ carries $N$ distinct codimension one tautological foliations $\mathcal G_1, \ldots, \mathcal G_N$,
defined on $\mathbb D^N$ by one the natural projections to $\mathbb D$. The foliations $\G_i$ are transversely hyperbolic
foliations on the complement of the singular points of $\mathbb D^{N}/\Gamma$.

\subsection{ Morphisms to irreducible polydisc quotients}\label{SS:preparation}
Let $\Gamma \subset \Aut(\D)^N$ be an irreducible lattice and
let $X$ be a complex compact manifold. Assume  there exists a morphism $\rho : X \to \mathbb D^N/\Gamma$
with image of postive dimension $p$.
Let $\F_1 = \rho^* \G_1, \ldots, \F_p= \rho^* \G_p$ be the pull-back to $X$ of $p$ among the $N$ tautological foliations
on $\mathbb D^N/\Gamma$ such that the foliations  are in general position.  i.e if $\omega_p$ is a local generator of $N_{ \F_i }^*$, $i=1,\ldots, p$, then $\omega_1\wedge \cdots \wedge\omega_p$ does not vanish identically.

The foliations $\F_i$ are all transversely hyperbolic foliations with finite quotient singularities. The
transverse hyperbolic structure is not necessarily defined  over codimension one components of the fibers of $\rho$ over $\sing(\mathbb D^N/\Gamma)$.

Let $D_j$ be the divisorial part of the foliation $\F_j$ and let $H_1,\ldots ,H_k$ some pairwise distinct prime divisors such that $D_j= \sum_{i=1}^k r_{ij} H_i$ where $r_{ij}\in \Q_{>-1}$. Let $L_j$ be the nef $\Q$ line bundle such that
\[
    N_{ \F_j }^*=L_j + D_j= L_j+ \sum_{i=1}^k r_{ij} H_i
\]

\begin{lemma}\label{L:preparation}
    The following assertions hold true:
	\begin{enumerate}
		\item\label{I:1} for any $j \in \{1, \ldots, p\}$
		\[
		N_{ \F_j }^*=L_j + \sum_{i=1}^k r_{ij} H_i
		\]
		where $L_j$ is a nef $\mathbb Q$ line bundle with $\nu(L_j)\geq 1$;
		\item\label{I:2} the hypersurface $H_i$ is $\F_j$ invariant whenever $r_{ij} \neq 0$;
		\item\label{I:3} the monodromies of the transversely hyperbolic structures of $\F_1, \ldots, \F_p$ around $H_i$ have all the
		same order $n_i=\mathrm{Min}\{m\in\N_{>0}|mr_{ij}\in\Z\}$.
        In particular, $r_{ij} \in \mathbb Z$ for some $j \in \{1, \ldots, p\}$
        if, and only if $r_{ij} \in \mathbb Z$ for all  $j \in \{ 1, \ldots, p\}$.
	\end{enumerate}
\end{lemma}
\begin{proof}
	Each of the foliations $\F_j$ is transversely hyperbolic outside its polar locus $H_j$.
    Fix some $j$. From \cite[Theorem 2]{Shimizu},  the local monodromies of the transversely hyperbolic structures of $\F_1, \ldots, \F_p$ around $H_j$  have all the same order,  which is non trivial. It follows  that the $n$ foliations $\F_1,\ldots,\F_p$ share the same polar locus.
	Once we have made this observation, the Lemma directly follows from  Proposition \ref{P:c1dec}.
\end{proof}

\subsection{A big divisor}
In our next statement, we keep the notation used in Lemma \ref{L:preparation}.

\begin{lemma}\label{L:big}
	Assume that $X$ is a projective manifold such that the morphism $\rho$ is generically finite (i.e $p= \mathrm{dim}\ X$).
    Then,  the divisor $L = \sum_{j=1}^p L_j$ is big.
\end{lemma}
\begin{proof}
	The nef $\mathbb Q$ divisor $L_j$ have Chern-Hodge classes in $H^1(X,\Omega^1_X)$ represented by
    semi-positive $(1,1)$-form $\eta_j$ obtained by pull-back of the Poincaré metric under distinguished first integrals.  Recall that the $(1,1)$-forms $\eta_i$ are  smooth outside the (common) polar locus of the tranverse strucures , and have $L_{loc}^1$ coefficients on $X$.
	
	Consequently, the Chern class of $\mathbb Q$ line bundle $L=\sum_{j=1}^n L_j$ is represented by a semi-positive $(1,1)$-form
	\[
	   \eta = \sum_{i=1}^p \eta_i
	\]
	with the same type of regularity and we can take its $p$-th power $\eta^p$ which is meant pointwise.
	
	Recall that on a $p$ dimensional complex compact manifold, the volume of a line bundle $E$ is defined as
	\[
        v(E)=\limsup \limits_{k\to \infty}\frac{p!}{k^n}h^0(X,kE) \, .
    \]
	 The line bundle is big, i.e has maximal Kodaira-Itaka dimension $\kappa (E)=p$ exactly when $v(E)>0$ and in that case,
    the lim sup is actually a true limit.  More generally, one can define the volume of a $\Q$ line bundle $E$ as  $v(L):= l^{-p}v(lE)$ where $l$ is a positive integer such that $lE$ is a line bundle. it is easily seen that this definition of volume does not depend of the choice of $E$.
	
	According to \cite[Theorem 1.2]{Boucksom2002}, $v(L)\geq\int_X \eta^p$.
    One concludes by noticing that the right hand side is positive. Indeed, $\eta^p$  restricts on $X-H$ to the smooth and non trivial semi-positive form $(p!) \eta_1 \wedge \cdots \wedge \eta_p $.
\end{proof}

\section{Numerical specialness for rank one subsheaves of $\Omega^1$}\label{numspecial}

\begin{prop}\label{P:corenew}
	Let $\Gamma \subset \Aut(\D)^N$ be an irreducible lattice and
	let $X$ be a compact complex  manifold.  Assume that there exists a morphism $\Psi : X \to \mathbb D^N/\Gamma$
	with image of positive dimension $p$ such that the pull-back $\F= \Psi^*\G$ of one of the tautological foliation $\G$,
    equipped with the pull-back transverse hyperbolic structure, has a $\Q$ effective divisorial part. Then there exists an invertible subsheaf $L\subset \Omega_X^p$  whose Kodaira dimension satisfies $\kappa (L)\geq p$
\end{prop}
\begin{proof}
	Let $V$ be a smooth projective model of the image of $\Psi$ determined by some birational morphism
    $\rho: V\to \mathrm{Im}\ \Psi\subset \D^N/\Gamma$.
	Retaining the notations of Subsection \ref{SS:preparation}, we have on $V$, $p=\mathrm{dim}\ V$ foliations in general position:
    $\mc H_j=\rho^*\G_j$. One can moreover assume that $\G_1=\G$.
	
	The conormal of each of these foliations splits as
	\[
	   N_{ \mc H_j }^*=L_j' + D_j'
	\]
	such that $D_j'$ is the divisorial part of $\mc H_j$.
	We also have  from Lemma \ref{L:big} that $\sum L_j'$ is a big $\Q$ divisor.
	
	Note also, using for instance Proposition \ref{P:behaviorsurj},
    that the divisorial part of $\F$ under pull-back by any birational morphism remains effective.
    Then, up to performing some blows up on $X$, one can suppose that $\rho$ factors through a dominant morphism $\varphi: X\to V$.
	
	Set $\mc F_j=\varphi^*\mc H_j$, $j=1,\ldots,p=\mathrm{dim}\ V$, so that $\F=\F_1$.
	Consider the decomposition
	\[
	   N_{ \F_j }^*=L_j + D_j= L_j+ \sum_{i=1}^k r_{ij} H_i
	\]
	as given in Subsection \ref{SS:preparation}.
	
	After renumbering the hypersurfaces $H_i$, one can assume  by Lemma \ref{L:preparation},
    we can assume existence of a integer $k' \le k$ such that
	$r_{ij} \in \Q - \Z$   for any $j \in \{ 1, \ldots, p\}$ if and only if $i\leq k'$. We can thus write, for any $j$,
	\[
	   \sum_{i=1}^k r_{ij} H_i = \sum_{i=1}^{k'} r_{ij} H_i + R_j
	\]
	where $R_j$ is an effective divisor. By assumption on the divisorial part of $\F$,  we have that $r_{i1}>0$ for every
	$i \in \{ 1, \ldots, k\}$. Therefore we can write
	\begin{equation}\label{E:nef}
	\sum_{i=1}^{k'} \left( p-1 + \sum_{j=1}^p r_{ij} \right) H_i = \sum_{i=1}^{k'} \left( r_{i1} +  \sum_{j=2}^p (1+ r_{ij}) \right) H_i \ge 0
	\end{equation}
	thanks to  \ref{SS:preparation}.
	
	Consider the morphism
	\begin{align*}
	   \sigma :  N^*_{\F_1} \otimes \cdots \otimes N^*_{\F_p} &\longrightarrow \Omega^p_X \\
	   \omega_1 \otimes \cdots \otimes \omega_n & \mapsto \omega_1 \wedge \cdots \wedge \omega_p .
	\end{align*}
	The saturation of the image of $\sigma$ is an invertible subsheaf $\mathcal O_X(F)  \subset \Omega^p_X$ isomorphic
	to
	\[
	   N^*_{\F_1} \otimes \cdots \otimes N^*_{\F_p} \otimes \mathcal O_X(\tang(\F_1, \ldots, \F_p)) ,
	\]
	where $\tang(\F_1, \ldots, \F_p)$ is the tangency divisor of the foliations $\F_1, \ldots, \F_p$.
	Since the hypersurfaces $H_1, \ldots, H_{k'}$ are invariant by all the foliations $\F_j$, it follows
	that $\tang(\F_1, \ldots, \F_p) \ge (p-1) \sum_{i=1}^{k'} H_i$.
	
	Another important point is that $L_j=\varphi^* (L_j')$ according to Proposition \ref{P:behaviorsurj}.
    From Lemma \ref{L:big} and because $\varphi$ is dominant, one deduces that the Kodaira dimension of
	$L=\sum L_j$ is at least $p$.
	Therefore
	\[
	   \displaystyle{  F \quad \ge  \quad
        \underbrace{\sum_{i=1}^{k'} \left( n-1 + \sum_{j=1}^p r_{ij} \right) H_i}_{\text{ $\mathbb Q$ effective according to Inequality (\ref{E:nef})}} \qquad   +  \underbrace{\sum_{j=1}^p L_j \, ,}_{\text{$\kappa( \sum_{j=1}^p L_j )\ge p $}} \, }
	\]
	and  $F$ can be written as the sum of a nef  $\mathbb Q$ line bundle $L$ of Kodaira dimension $\geq p$
	and an effective $\mathbb Q$ divisor. It follows that $\kappa (F)\geq p$. 
\end{proof}
Let $X$ a $n$ dimensional compact K\"ahler manifold equipped with a codimension one foliation $\F$ whose conormal sheaf $N_\F^*$ is pseudo-effective. Denote by $\nu$ (resp. $\kappa$) the numerical (resp. Kodaira) dimension of $N_\F^*$.  Theorem \ref*{TH:conormalvshyperbolic} guarantees that the Zariski decomposition of $c_1(N_\F^*)$ reads as
\[
    c_1(N_\F^*)=N +Z
\]
where $N$ is a $\Q$ effective divisor such the intersection matrix $(N_i \cdot N_j \cdot \Theta^{n-2})$  is negative definite, where the $N_i$'s are the irreducible components of $\Supp(N)$, $Z$ is a \textit{nef} class and $\Theta$ any K\"ahler class.

According to  \cite[Theorem 4] {MR3644247},  there are three possible cases according to the value of $\nu$ (defined by $Z^\nu\not=0, Z^{ \nu+1}=0$)
\begin{enumerate}
	\item $\nu=0=\kappa$
	\item $\nu=1=\kappa$\label{I:algint}
	\item $\nu=1$, $\kappa=-\infty$ ( the  non abundant case )\label{I:dense}
\end{enumerate}

The two last cases also strongly differ from the dynamical viewpoint (see  \cite[loc;cit]{MR3644247}): in the first situation \ref{I:algint},
$\F$ is algebraically integrable, in the second one \ref{I:dense}, the foliation is quasi minimal (all leaves, exceptely finitely many are  dense for the euclidean topology). We are interested in the last situation where abundance does not hold.

\subsection{Proof of Theorem \ref{THM:D}}

	Assume for a while that $X$ is \textit{projective}. By \cite[Theorem 6]{MR3644247} and also taking into account Remark \ref{R:KLT},
    there exists a morphism $\Psi: X\to \D^N/\Gamma$ whose image has dimension $\geq 2$ such that $\F=\Psi^*\G$ where $\G$ is one of the tautological  foliation on $\D^N/\Gamma$.
    Moreover, the tranverse hyperbolic structure of $\F$ as described in Theorem \ref*{TH:conormalvshyperbolic} is
    obtained by pulling-back via $\Psi$ of the natural transverse hyperbolic structure of $\G$. This is implicitely proven in  (\cite[Section 6]{MR3644247}, especially p. 22,23) but one can also invoke the uniqueness property stated in Proposition \ref{P:uniqueness}.
	
	Let us now consider the general case of compact K\" ahler manifolds.
    Up to renumbering the components $N_i$ of the negative part $N$,
    one can assume that for some $q\in \N$, $\lambda_i\in \Q-\N$ for $i=1, \ldots,q$, $\lambda_i\in \N_{>0}$ for $i>q$. Set $N'= \sum_{i=1}^q \lambda_i N_i$.  Let $\rho: \pi_1 (X-\mathrm{Supp}(N'))\to \Aut (\D)$ the monodromy representation of the transverse hyperbolic structure.  According to \cite[Proposition 4.6]{MR3644247}, the image of $\rho$ is Zariski dense (actually dense in the euclidean topology). By Selberg's Lemma, there exists a finite index torsion free normal subgroup in the image of $\rho$. This enables to construct a finite Galoisian   cover $R: \hat X\to X$  with  branch locus $\mathrm{Supp}(N')$ such that the pull-back representation $R^*\rho$ is actually well defined as a morphism $\pi_1(\hat X)\to \Aut (\D)$ with torsion free image. According to \cite[Theorem 1]{Vajaitu96}, $\hat X$ is a K\"ahlerian analytic space and then, by Hironaka \cite{Hironaka77} admits a resolution of singularites which is a compact K\"ahler manifold $Y$. We have thus construct a surjective morphism with generically finite fibers $\psi:Y\to X$ between compact K\"ahler manifolds such the pull-back foliation $\F_Y:=\psi^*\F$ is transversely hyperbolic without poles. The associated monodromy representation is nothing but $\rho_Y:=\psi^*\rho: \pi_1(Y)\to \Aut(\D)$ with dense and torsion free image. Note that one can prove than $\rho_Y$ (and equivalently $\rho$) has Zariski dense image in $\Aut(\D)$ without resorting to \cite[Proposition 4.6]{MR3644247}. Indeed, let $\tilde Y$ be the universal covering of $Y$. Denote by $f: \tilde Y\to \D$ be the $\rho$ equivariant holomorphic obtained by developping the transverse hyperbolic structure of $\F_Y$. Because $Y$ is K\" ahler, $f$ is also harmonic and the Zariski density of the image of $\rho$ follows  from \cite{Corlette88,Labourie91} (regarding $\D$ as a symmetric space of the non compact type).

    By \cite{MR1384908, MR3314830} and up to taking a bimeromorphic smooth model of $Y$,  $\rho_Y$ factors through $\rho_V:\pi_1(V)\to \Aut (\D)$  via a surjective morphism $e:Y\to V$ with connected fibers, where $V$ is a projective manifold of the general type.

    In particular $e$ factors through the algebraic reduction map ${red}_Y : Y\to {\Red} (Y) $ (here and henceforth, we will assume, up to taking appropriate smooth models, that all algebraic reduction spaces are projective manifolds and all reduction maps are morphisms). By Lemma \ref{L:foldescent}, there exists on   ${\Red} (Y)$ a transversely hyperbolic foliation $\F_1$ such that $\F_Y= {red}_Y^*\F_1$, and such the monodromies representations factor accordingly.

    Observe also that the group of deck transformations of the Galoisian cover $\hat X\to X$ induces on $Y$ a finite group of bimeromorphic transformations $G$ preserving the foliation $\F_Y$. Consider the action $G\times \C (Y)\to\C (Y)$ defined by $g \cdot f=f\circ g^{-1} $ and denote by $K$ the kernel of this action. By the very definition of ${\Red} (Y)$, $G$ acts on $Y$ by preserving the fibration of the reduction map and the induces a faithful action of $G_1:=G/H$ on $ {\Red} (Y)$ by birational transformations also preserving the foliation $\F_1$. By construction, observe also  that the field of rational functions of ${\Red} (X)$ is precisely $\C(X)= { \C (Y)}^G$. This implies that there exists on ${\Red} (X)$  a foliation, $\F_2$ such that $\F_1=r_{Y,X}^*\F_2$ where $r_{Y,X}:{\Red} (Y)\to {\Red} (X)$ is the rational map induced by the inclusion $\C(X)\subset \C (Y)$ and then makes the following diagram commute, up to replacing $Y$, $\Red(Y)$, $\Red(X)$, and $V$ by suitable non singular models.
    \begin{center}
    \begin{tikzpicture}
    \matrix (m) [matrix of math nodes,row sep=3em,column sep=4em,minimum width=2em]
    {
	( Y,\F_Y) & ( \Red(Y),\F_1)&V \\
	( X,\F) & ( \Red (X),\F_2)& \\};
    \path[-stealth]
    (m-1-1) edge node [left] {$\psi$} (m-2-1)
    edge  node [above] {${ red}_Y$} (m-1-2)
    (m-1-2)edge node [above]{$e'$} (m-1-3)
    (m-1-1) edge [bend left] node [above] {$e\;$} (m-1-3)
    (m-2-1.east|-m-2-2)
    edge node [above] {${red}_X$} (m-2-2)
    (m-1-2) edge [dashed] node [right] {${r}_{Y,X}$} (m-2-2);
    \end{tikzpicture}
    \end{center}

    In particular, we have that $\F= {red}_X^*\F_2$. According to Lemma \ref{L: Pullbackbehavior}, $\F_2$ admits a transversely hyperbolic structure with quotient singularities compatible with that of $\F$. Applying Theorem  \ref{THM:D} in the projective case, one deduces that $\F$ is obtained by pull-back of a tautological foliation on a polydisk quotient also compatible with the transverse hyperbolic structure (Proposition \ref{P:uniqueness}). The proof is thus complete.
\qed

\begin{thm}\label{T:firsthalf}
	Let $(X,\F)$ a foliated compact K\" ahler manifold with $\mathrm{codim}(\F)=1$.
    Assume that $N_\F^*$ is pseudo-effective and not abundant in the sense defined above. Then there exists an invertible subsheaf $L\subset \Omega_X^p$ for some $p\geq 2$ such that $\kappa (L)=p$.
\end{thm}
\begin{proof}
    The decomposition provided by Proposition \ref{P:c1dec} takes the form
    \begin{equation}\label{E:decpsef}
        N_\F^*= L+D
    \end{equation}
    where the divisorial part is \textit{effective}.  This property is clealy invariant under pull-back by bimeromorphic morphisms, so that one can apply Theorem  \ref{THM:D}.  From Proposition \ref{P:corenew}, one derives the existence of an invertible subsheaf $L\subset \Omega_X^p$ for some $p\geq 2$ such that $\kappa (L)\geq p$ hence $=p$ by Bogomolov's upper bound.
\end{proof}	

\begin{rem}
    In general, one cannot deduce non-specialness of $X$ just from the existence of a proper morphism to an irreducible quotient of a polydisc.
    Indeed in \cite[Theorem 12.1]{Gra02} Granath produces examples of rational and $K3$-surfaces obtained as minimal resolutions of singular compact quotients $\mathbb D^2/\Gamma$.
\end{rem}

\begin{rem}
    Another way to interpret the preceding proof is the following. Considering as before $V$ a smooth projective model of the image of $\Psi$, $V$ is equipped with a natural divisor $\Delta=\sum_i (1-\frac{1}{m_i})D_i$ supported over $\Sing (\D^N/\Gamma)$. It follows from \cite{CDG19} that $K_V+\Delta$ is big. As above, up to performing some blowups on $X$, we have a dominant morphism $\varphi: X\to V$. Then it follows from the proof of Theorem~\ref{T:firsthalf} that $\varphi: X\to (V, \Delta)$ is an \emph{orbifold morphism} in the sense of Campana i.e. $\varphi$ ramifies over $D_i$ with multiplicity at least $m_i$. This implies that $\varphi^*(K_V+\Delta) \subset \Omega_X^p$ is a Bogomolov sheaf where $p:=\dim V$.
\end{rem}

\subsection{Proof of Theorem \ref{THM:A}}
    The result is obvious  if $\kappa(L)>0$ (hence $=1$). Otherwise, this is Theorem \ref{T:firsthalf}.
\qed

\section{Entire curves}\label{ec}

First, remark that in the case  $\nu=1=\kappa$, we have a Bogomolov sheaf $L\subset \Omega_X$ which corresponds (see \cite{Ca04}) to a fibration of general type $F: X \to C$ onto a curve. This means that the orbifold base of the fibration $(C, \Delta)$ is of general type. It is now a classical fact \cite{Nev70} that in this setting, for any entire curve $f:\C \to X$, $F\circ f: \C \to C$ has to be constant. Therefore $f:\C \to X$ cannot be Zariski dense.

So, now we deal with the non-abundant case $\nu=1$, $\kappa=-\infty$.

\begin{thm}\label{degen}
    Let $(X,\F)$ a foliated compact K\" ahler manifold with $\mathrm{codim}(\F)=1$. Assume that $N_\F^*$ is pseudo-effective and not abundant.
	Then any entire curve $f: \C \to X$ is algebraically degenerate i.e., $f(\C)$ is not Zariski dense.
\end{thm}

We shall start with a lemma.

\begin{lemma}\label{entirecurve}
	Let $(X,\F)$ a foliated compact K\" ahler manifold with $\mathrm{codim}(\F)=1$. Assume that $N_\F^*$ is pseudo-effective and not abundant.
	Then any entire curve $f: \C \to X$ is tangent to $\F$.
\end{lemma}
\begin{proof}
    From \cite{MR3644247} Theorem 3.1 (see also Theorem \ref{TH:conormalvshyperbolic}), we have a singular transverse metric $h$ with curvature current $\Theta_h=-(h+[N])$. It is a smooth transverse metric of constant curvature $-1$ on $X\setminus (\Sing(\F) \cup \Supp N)$.
    Suppose $f: \C \to X$ is not tangent to $\F$. In particular, $f(\C) \not\subset \Sing(\F) \cup H$. Therefore $f^*h$ induces a non-zero singular metric $\gamma(t)=\gamma_0(t)i\,dt\wedge d\overline{t}$ on $\C$ where $-\Ric \gamma \geq \gamma$ in the sense of currents.  But the Ahlfors-Schwarz lemma (see \cite{De97} Theorem 3.2) implies that $\gamma \equiv 0$, a contradiction.
\end{proof}

\subsection{Proof of Theorem \ref{THM:B}}
	The subsheaf $L$ determines a foliation $\F$ whose conormal bundle $N_\F^*$  has numerical dimension $\nu=1$.  According to the above discussion, it suffices to consider the case $\kappa=\kappa (N_\F^*)=-\infty$. From the preceding lemma, we can suppose that $f: \C \to X$ is tangent to $\F$. By 
	Theorem \ref{THM:D} (up to replacing $X$ by a smooth model) there exists a morphism $\Psi: X\to \mf H:=\D^N/\Gamma$ such that
    $\F=\Psi^*\G$ where $\G$ is one of the tautological  foliation on $X$. Therefore $\Psi(f): \C \to \mf H$ is tangent to $\G$ and is constant thanks to the hyperbolicity of the leaves on by $\mf H$ \cite{rousseautouzet15} Proposition 3.1 . This concludes the proof.
\qed

\section{Geometric specialness}\label{geospecialness}
In this section, we prove Theorem \ref{THM:C}. We will start by proving in our setting a particular case of Lang-Vojta's conjecture.

\begin{conj}[Lang-Vojta]
Let $X$ be a projective variety of general type and $L$ an ample line bundle. Then there is a proper algebraic subset $Z \not \subset X$ and a constant $\alpha$, such that for every smooth projective connected curve $C$ and every morphism $f:C\to X$ with $f(C) \not \subset Z$, one has
	  $$\deg f^*L \leq \alpha (2g(C)-2).$$
\end{conj}

\

Here, we prove the following particular case.
\begin{prop}\label{degreebound}
	Let $\Gamma \subset \Aut(\D)^N$ be an irreducible lattice and
	let $X$ be a  complex projective  manifold.  If there exists a generically finite morphism $\Psi : X \to \mathbb D^N/\Gamma$
	 such that the pull-back of one of the tautological foliation (equipped with the pull-back transverse hyperbolic structure)
	  has a $\Q$-effective divisorial part, then there exists a big line bundle $L$ on $X$, a proper algebraic subset $Z \not \subset X$ and a constant $\alpha$, such that for every smooth projective connected curve $C$ and every morphism $f:C\to X$ with $f(C) \not \subset Z$, one has
	  $$\deg f^*L \leq \alpha (2g(C)-2).$$
\end{prop}

\begin{proof}
Using the same notations as in Section \ref*{S:tautological}, we consider the decomposition
	\[
	   N_{ \F_j }^*=L_j + D_j= L_j+ \sum_{i=1}^k r_{ij} H_i.
	\]
	
The morphism $f: C \to X$ induces a morphism $f': C \to \mb P(T_X)$ which implies the algebraic tautological inequality
$$\deg_C(f'^*(\mathcal{O}(1)) \leq 2g(C)-2.$$
From \cite{rousseautouzet15}, we know that if $\Psi(f(C))$ is not constant then $f(C)$ is not contained in a leaf of any foliation $\F_j$. We take $Z$ to be the union of the positive dimensional fibers of $\Psi$ and the $H_i$ and assume that $f(C)$ is not contained in $Z$.

To the foliation $\mathcal{F}_i$ is associated a divisor $Z_i \subset \mb P(T_X),$ linearly equivalent to $\mathcal{O}(1)+{N}_{\mathcal{F}_i}.$ Then the algebraic tautological inequality gives
$$\deg_C(f^*({N}^*_{\mathcal{F}_i})) \leq \deg_C(f'^*(Z_i))+ \deg_C(f^*({N}^*_{\mathcal{F}_i})) \leq 2g(C)-2.$$
The first inequality comes from the non-tangency of the algebraic curve with the foliation which implies $0 \leq \deg_C(f'^*(Z_i))$.
By assumption on the divisorial part of $\F_1$,  we have that $r_{i1}>0$ for every
	$i \in \{ 1, \ldots, k'\}$. Therefore
	$$\deg f^*L_1+	\sum_{i=1}^{k'} r_{i1} \deg f^*H_i \leq 2g(C)-2.$$
Since $L_1$ is nef, we obtain $\deg f^*H_i \leq \frac{1}{r_{i1}}(2g(C)-2)$ for all $i \in \{ 1, \ldots, k'\}$.
Let $L:=\sum_j L_j$, then $L$ is big according to Lemma \ref{L:big} and the previous inequalities give
$$\deg f^*L\leq \sum_j \deg f^* N_{ \F_j }^*-\sum_j \sum_{i=1}^{k'} r_{ij} \deg f^* H_i \leq (2g(C)-2)(p+\sum_j \sum_{i=1}^{k'}  \frac{|r_{ij}|} {r_{i1}}). $$

\end{proof}

\begin{cor}\label{geospecial}
Under the same assumptions, $X$ is not geometrically special.
\end{cor}

\begin{proof}
Suppose $X$ is geometrically special.
Consider the Zariski open set $U:=X \setminus Z$. Then there exists a smooth projective connected curve $C$, a point $c$ in $C$, a point $u$ in $U$, and a sequence of  morphisms $f_i:C\to X$   with $f_i(c) = u$ for $i=1,2,\ldots$ such that $C\times X $ is covered by the graphs $\Gamma_{f_i}\subset C\times X$ of these maps. From the previous Proposition \ref{degreebound}, all these pointed maps have bounded degree. Therefore by Bend-and-Break, we obtain a rational curve passing through $u$. Such a curve has to be tangent to the foliation $\mathcal{F}_1$ by Lemma \ref{entirecurve}. This gives a contradiction.
\end{proof}

\subsection{Proof of Theorem \ref{THM:C}}

As in the proof of \ref{THM:B}, consider the foliation $\F$ associated to $L$. In the case  $\nu=1=\kappa$, we have a Bogomolov sheaf $L\subset \Omega_X$ which corresponds (see \cite{Ca04}) to a fibration of general type $F: X \to D$ onto a curve. This means that the orbifold base of the fibration $(D, \Delta)$ is of general type. This implies finiteness of orbifold morphisms $f: C \to (D, \Delta)$ (Theorem 3.8 \cite{Ca05}) and therefore $X$ cannot be geometrically special in this case.

In the non-abundant case $\nu=1$, $\kappa=-\infty$, there exists a morphism $\Psi: X\to \mf H:=\D^N/\Gamma$ such that $\F=\Psi^*\G$ where $\G$ is one of the tautological  foliation on $X$. Consider $V$ a smooth projective model of the image of $\Psi$ given by some generically finite map $V \to \mathbb D^N/\Gamma$. Let $\phi: X \to V$ be the induced dominant morphism (modulo some blow-ups on $X$).

Let $f: C \to X$ be a morphism. Following the same proof as in Proposition \ref{degreebound}, we obtain $ \deg (\phi\circ f)^*L'= \deg f^*L \leq \alpha (2g-2)$.
Since $L'$ is a big line bundle on $V$, the same proof as in Corollary \ref{geospecial} applied to the sequence $\phi \circ f_i$ gives that $X$ is not geometrically special.
\qed

\begin{rem}
In \cite{JR}, it is proved that if there exists a Zariski-dense representation $\rho: \pi_1(X) \to G(\C)$ ($G$ an almost simple algebraic group), then $X$ is not geometrically special. We cannot apply this result here since the monodromy representation $\rho: \pi_1(X_0) \to \Aut(\D)$ is \textit{a priori} defined only on $X_0:= X \setminus \Supp N$ where $N$ is the negative part of $c_1(N_\F ^*)$.
\end{rem}

\section{Higher codimensions}\label{highcodim}

As recalled in the introductive part, the  existence of a rank one coherent subsheaf $L$ of $\Omega_X^p$ having \emph{numerical} dimension $p$  and such that $\mathrm{codim}\ (\mathrm{Ker}\ L)=p$  on a compact K\" ahler  manifold $X$  can be translated into the existence of a codimension $p$ foliation  $\F$ on $X$ such that $\mathrm{det}\ N_\F^*$ (which is somehow the canonical sheaf of the "space of leaves" $X/\F$) has numerical dimension $p$. We do not know how to generalize Theorem \ref{THM:A} to  codimension $p>1$, in particular because we do not have at our disposal sufficiently precise structure results for this category of foliations, unlike in the case $p=1$. However, it remains possible to reach the same conclusion under strong assumptions on the subsheaf of $\Omega^p_X$. For instance, as stated in Theorem \ref{THM:E} in the Introduction, this is the case if we assume that the subsheaf of $\Omega^p_X$ defines a smooth foliation with conormal bundle having Chern class represented by a smooth $(1,1)$-form with semi-positive curvature of constant rank  $p$.

\begin{proof}[ Proof of Theorem E]
 We first assert that $\eta$ is basic for $\F$. More generally, one has the more general phenomenon, as proved by Demailly \cite{ MR1922099}: if $X$ is compact K\"ahler $\omega\in H^0(X, \Omega^p \otimes L)$ such that $L$ is a line bundle whose dual $L^*$ is pseudo-effective, then $\Theta\wedge \omega=0$ where $\Theta$ is any closed positive current representing $c_1(L^*)$.

 In our situation this implies that $\eta$ defines a holonomy transverse invariant K\" ahler metric for $\F$. Equivalently, the kernel of $\eta$ is exactly the tangent bundle to $\F$.  One thus inherits  another real basic $(1,1)$-form, namely the transverse Ricci form $r=-\mathrm{Ricci}(\eta)$. In some holomorphic coordinates $(z_1,\ldots,z_p)$ parameterizing the local space of leaves, it reads as $$\eta= {\sqrt{-1}} \sum_{i,j}g_{ij}dz_i \wedge d\bar{z_j}$$
where $g_{ij}$ depend only of the transverse variables $(z_1,\ldots,z_p)$
and $$r= -\frac{\sqrt{ -1}}{\pi}\partial\bar{\partial}\log \left(\frac{\eta^p}{{|dz_1\wedge\cdots \wedge dz_p|}^2}\right).$$
Note that $-r$ also represents $c_1(N_\F^*)$, so that there exists, by the $dd^c$ Lemma, a smooth function $f:X\to \R$ such that $-r=\eta+dd^c f$. Note that $dd^c f$ is basic as it is a sum of two basic forms.  This implies that $f$ is pluriharmonic along the leaves of $\F$. It turns out that $f$ is basic: indeed, let $\mathcal L$ be a leaf and $\overline{\mathcal L}$ its topological closure. Let $x\in\overline{\mathcal L}$ such that $\restr{f}{\overline{\mathcal L}}$ reaches its maximum at $x$ and let ${\mathcal L}_x$ be the leaf passing through $x$. By the maximum principle for pluriharmonic functions $f$ is constant on ${\mathcal L}_x$, hence on $ \overline{ {\mathcal L}_x}$. On the other hand, the leaves closure form a partition of $X$,  a common feature for Riemannian foliations \cite[Theorem 5.1]{ MR932463}. In particular, $\overline{\mathcal L}=
\overline{ {\mathcal L}_x}$. As the original leaf $\mathcal L$ has been chosen arbitrarily,  this enables to conclude that $f$ is leafwise constant, as wanted. Then, $r$ and $\eta$ are not only cohomologous in the ordinary $\partial\bar{\partial}$ cohomology, but also in the \textit{basic} $\partial\bar{\partial}$ cohomology. By the foliated version of Yau's solution to Calabi conjecture, due to El Kacimi  \cite[Section 3.5]{MR1042454}, there exists for $\F$ an invariant transverse K\"ahler metric whose Ricci form is equal to $-\eta$.

If the leaves of $\F$ are closed,  the space of leaves $X/\F$ is naturally equipped with a structure of a K\"ahler orbifold (in the usual sense). The $(1,1)$-form $\eta$ descends on  $X/\F$ as a positive $(1,1)$-form representing the Chern class of the (orbifold) canonical bundle $K_{X/\F}$.  The orbifold analogue of Kodaira embedding theorem \cite[Section 7]{Baily57} enables to exhibit a Bogomolov sheaf on $X$, thus proving that $X$ is not special.

Otherwise, if the leaves of $\F$ are not closed,  the strategy consists in producing  a representation $\pi_1(X)\to G$ with dense image, where $G$ is a real semi-simple algebraic group  and apply Zuo's Theorem \cite{MR1384908} or \cite[Theorem 1]{MR3314830} to conclude: if there exists such a representation, up to replacing $X$ by a finite étale cover $\tilde X$, there exists a meromorphic fibration $f: \tilde X\to V$ where $V$ is a projective manifold of general type (through which the representation factorizes).

We now explain how to produce such a representation.
Under our assumptions,  $\F$ is a transversely K\" ahler foliation $\F$ whose leaves are not closed and whose transverse Ricci curvature is negative, that is $\eta$ is semi-negative with constant rank equal to $p$ the complex codimension of $\F$.  This corresponds to the monodromy of the so called commuting sheaf $\mathcal C$. This sheaf is a locally constant sheaf of  Lie algebras $\mathfrak g$ of basic Killing  vector fields which encodes the dynamic of $\F$ and defined in the general setting of Riemannian foliations (see   \cite[Section 5.3]{ MR932463}). Under our negativity assumption, one can prove \cite[Théor\`eme 1.1]{MR2790821}, that $\mathfrak g$ is semi-simple.  This implies that the image of representation $\alpha:\pi_1(X)\to \mathrm{Aut}(\mathfrak g)$ (={ real semi-simple algebraic group}) intersects  $\mathrm{Aut}^0(\mathfrak g)$ as a dense subgroup. Indeed, the closure of the image of $\alpha$ contains the image $\mathrm{Ad}\ g$ of the exponential of the adjoint representation of $\mathfrak g$ according to \cite[E. Salem Appendix, Proposition 3.7]{MR932463}. We can thus apply Zuo's Theorem to produce the sought fibration $\tilde X\mapsto V$, thus proving that $X$ is not special.
\end{proof}
\begin{rem}
In \cite{Mok2000}, Mok considered compact K\" ahler manifolds $X$ equipped with a $d$-closed holomorphic one form twisted by a locally constant bundle of Hilbert spaces $E_\Phi$. This defines on $X$ a foliation which is transversely Riemannian on a dense open subset of $X$ and whose transverse metric (semi-K\"ahler structure in the language of \textit{loc.cit}) is cooked up from an orthonormal basis on the typical fiber of $E_\Phi$ and whose "Ricci curvature" carries some negativity properties. Under some circumstances, Mok shows the existence of fibration of (an \'etale cover of) $X$ onto varieties of general type. It could be tempting to relate the existence of a numerical Bogomolov sheaf to the existence of a semi-K\" ahler structure arising from twisted forms. For instance, in codimension one, transversely hyperbolic foliations can be regarded  as foliations defined by the kernel of a $E_\Phi$ valued closed holomorphic one form where $E_\Phi\to X$ is the locally constant bundle of Hilbert spaces arising from a unitary representation $\pi_1(X) \to U(H)$, where $H$ is the space of square integrable antiholomorphic forms on the disk  (\cite[Section 4]{Mok97}).
\end{rem}

\bibliography{references}
\bibliographystyle{alpha}

\end{document}